\documentclass[11pt]{amsart}

\usepackage{amssymb,latexsym}
\usepackage{enumitem}
\usepackage{comment}
\usepackage{bbm}
\usepackage{bm}

\makeatletter

\@namedef{subjclassname@2010}{

  \textup{2010} Mathematics Subject Hlassification}
\usepackage{amssymb, amsmath,amsthm}
\usepackage{hyperref}
\newtheorem{thm}{Theorem}[section]

\newtheorem{lem}[thm]{Lemma}
\theoremstyle{definition}

\newcommand{\mc}{\mathcal}

\newcommand{\mb}{\mathbb}

\newcommand{\N}{\mb{N}}

\newcommand{\lla}{\left\langle}
\newcommand{\e}{\varepsilon}

\newcommand{\rra}{\right\rangle}

\newcommand{\lb}{\label}
\newcommand{\ef}{\eqref}

\newenvironment{acknowledgements} {\begin{abstract}} {\end{abstract}}

\newcommand{\proofpart}[2]{%
  \par
  \addvspace{\medskipamount}%
  \noindent\emph{Part #1: #2}\par\nobreak
  \addvspace{\smallskipamount}%
  \@afterheading
}

\title{On Correlations of Liouville-like functions}
\author{Yichen You}
\address{Department of Mathematical Sciences, Durham University, Stockton Road, Durham, DH1 3LE, UK}
\email{yichen.you@outlook.com}

\begin{document}
\begin{abstract}
Let $\mc{A}$ be a set of mutually coprime positive integers, satisfying
\begin{align*}
    \sum\limits_{a\in\mc{A}}\frac{1}{a} = \infty.
\end{align*}
Define the (possibly non-multiplicative) "Liouville-like" functions 
\begin{align*}
   \lambda_{\mc{A}}(n) = (-1)^{\#\{a:a|n, a \in \mc{A}\}} \text{ or } (-1)^{\#\{a:a^\nu\parallel n, a \in \mc{A}, \nu \in \N\}}. 
\end{align*}
We show that 
\begin{align*}
    \lim\limits_{x\to\infty}\frac{1}{x}\sum\limits_{n \leq x} \lambda_\mc{A}(n) = 0
\end{align*}
holds, answering a question of de la Rue. 

We also show that if $\mc{A}\cap \mb{P}$ has relative density $0$ in $\mb{P}$, the $k$-point correlations of $\lambda_{\mc{A}}$ satisfies
\begin{align*}
    \lim_{x \to \infty}\frac{1}{x}\sum\limits_{n\leq x}\lambda_\mc{A}(a_1n+h_1)\cdots\lambda_\mc{A}(a_kn+h_k) = 0,
\end{align*}
where $k \geq 2, a_ih_j \neq a_jh_i$ for all $1 \leq i < j \leq k$,
extending a recent result of O. Klurman, A. P. Mangerel, and J. Teräväinen.
\end{abstract}
\maketitle

\section{Introduction}

Let $\mc{K} \subset \N - \{1\}$ be a set of mutually coprime positive integers. We define $\omega_\mc{K}(n)$ and $\Omega_\mc{K}(n)$ by
\begin{align*}
    \omega_\mc{K}(n) = \sum\limits_{\substack{k | n, \\ k \in \mc{K}}} 1, \;\;\; \Omega_\mc{K}(n) = \sum\limits_{\substack{k^\nu || n, \\ k \in \mc{K}}} \nu.
\end{align*}
In the following, we will use $\mc{A}$ to denote a subset of $\N - \{1\}$ satisfying 
\begin{align} \lb{con1}
    (a, b) = 1 \text{ for all distinct } a, b \in \mc{A}, \text{ and } \sum\limits_{a\in \mc{A}} \frac{1}{a} = \infty.
\end{align}
Let 
\begin{align}\lb{deflambda}
    \lambda_{\mc{A}}(n)= (-1)^{\omega_\mc{A}(n)}\text{ or }(-1)^{\Omega_\mc{A}(n)}
\end{align}
be a "Liouville-like" function. As $\mc{A}$ is a sufficiently dense sequence because of \ef{con1}, we might expect that $\omega_\mc{A}(n)$ and $\Omega_\mc{A}(n)$ grow quickly for typical $n$, which means that $\lambda_{\mc{A}}(n)$ should change sign quite frequently. Motivated by a problem posed by de la Rue at the AIM conference on Sarnak’s conjecture (Problem 7.2 in the list \cite{delarue}), we investigate for which $\mc{A}$ the correlation averages of $\lambda_{\mc{A}}(n)$ satisfy the \textit{Chowla property}, i.e., 
\begin{align}\lb{chowla}
    \lim_{x \to \infty}\frac{1}{x}\sum\limits_{n\leq x}\lambda_\mc{A}(n+h_1)\cdots\lambda_\mc{A}(n+h_k) = 0
\end{align}
for any $k \geq 1$ and $h_1, ..., h_k\in \N$ fixed and distinct. In the special case that $\mc{A} =\mb{P}$ the set of all primes, \ef{chowla} coincides with Chowla's conjecture when $\lambda_{\mc{A}}(n)= (-1)^{\Omega_\mc{A}(n)}$ is the Liouville function. 

Note that if $\mc{A}$ is not a subset of primes or prime powers then $\lambda_{\mc{A}}$ is not multiplicative\footnote{Let $f: \N \to \mathbb{C}$. $f$ is multiplicative if $f(ab) = f(a)f(b)$ whenever $a, b \in \N, (a, b) = 1$.}. For example, let $p_1, p_2$ be two distinct primes and $p_1p_2 \in \mc{A}$ (so that $p_1, p_2 \notin \mc{A}$). We have
\begin{align*}
    -1 = \lambda_{\mc{A}}(p_1p_2) \neq \lambda_{\mc{A}}(p_1)\lambda_{\mc{A}}(p_2) = 1.
\end{align*}
Recently, O. Klurman, A. P. Mangerel, and J. Teräväinen \cite{klurman2023elliott} have shown that \ef{chowla} holds if $\mc{A}$ is a set of primes which has relative density $0$ \footnote{The relative density of a subset $\mc{P} \subset \mb{P}$ within $\mb{P}$ is defined as $\lim\limits_{x \to \infty} |\mc{P}\cap[1, x]|/|\mb{P}\cap[1, x]|$, provided the limit exists.} in $\mb{P}$. It turns out that their result can be generalized to a larger class of sets $\mc{A}$ satisfying \ef{con1}. 
Our main result is
\begin{thm}\lb{thm1.2}
  Let $\mc{A} \subset \N$ satisfy \ef{con1} and $\lambda_\mc{A}(n) = (-1)^{\omega_\mc{A}(n)}$ or $(-1)^{\Omega_\mc{A}(n)}$.
  \begin{itemize}
      \item[(a)] We have
      \begin{align}\lb{eq1}
         \lim\limits_{x\to\infty}\frac{1}{x}\sum\limits_{n \leq x} \lambda_\mc{A}(n) = 0.
     \end{align}
     \item[(b)] Let $k \geq 2$ and $a_1, ..., a_k, h_1, .., h_k\in \N$ fixed with $a_ih_j-a_jh_i\neq 0$ whenever $i \neq j$. Then we have
     \begin{align*}
         \lim_{x \to \infty}\frac{1}{x}\sum\limits_{n\leq x}\lambda_\mc{A}(a_1n+h_1)\cdots\lambda_\mc{A}(a_kn+h_k) = 0,
     \end{align*}
     provided that $\mc{A} \cap \mb{P}$ has relative density $0$ in $\mb{P}$.
      
  \end{itemize}
\end{thm}

\section{Proof of Theorem \ref{thm1.2}}

We decompose $\mc{A}$ into two disjoint parts as
\begin{align}\lb{decom}
    \mc{A} = \mc{C}\sqcup \mc{P}
\end{align}
where $\mc{C}$ consists of all composite numbers in $\mc{A}$ and $\mc{P} = \mc{A} \cap \mb{P}$. Since the elements in $\mc{A}$  are mutually coprime, in Lemma \ref{lem1.1} we will show that $\mc{C}$ is sparse and $\mc{P}$ is dense, which is crucial to deduce Theorem \ref{thm1.2}. 

\textbf{Notation:} Let $t > 0$ and $\mc{K} \subset \N$. We use $\mc{K}_t$ to denote $\mc{K} \cap [1, t]$ and $\lla\mc{C}\rra$ to denote the set
\begin{align*}
    \lla\mc{C}\rra = \{n:  n =  \prod\limits_{c \in \mc{C}}c^\nu &\text{ with } \nu \in \N\cup\{0\} \}.
    \end{align*}
We will keep using these notations in the rest of the paper. 
\begin{lem}\lb{lem1.1}
    Let $\mc{A} \subset \N-\{1\}$ satisfy \ef{con1}. Then we can decompose $\mc{A} = \mc{C}\sqcup \mc{P}$, such that
    \begin{align}\lb{sparse}
        \sum\limits_{a\in \mc{C}} \frac{1}{a} < \infty\text{ and }\sum\limits_{a\in \mc{P}} \frac{1}{a} = \infty.
    \end{align}
    Moreover, we have
    \begin{align}\lb{index}
        |\lla\mc{C}\rra_x| \leq x^{1/2}.
    \end{align}    
\end{lem}
\begin{proof}
Let $\mc{A} \subset \N-\{1\}$ satisfy \ef{con1}. Then
\begin{align*}
    \sum\limits_{a\in \mc{A}} \frac{1}{a} =  \sum\limits_{a\in \mc{P}} \frac{1}{a} + \sum\limits_{a\in \mc{C}} \frac{1}{a} = \infty.
\end{align*}
Let $p_a$ be the smallest prime $p_a | a$ for $a \in \mc{A}$ and $\mc{P}_\mc{C} = \{p_a: a \in \mc{C}\}$. Observe that all primes $p_a\in\mc{P}_\mc{C}$ are distinct since all elements in $\mc{C}$ are coprime. Let $a \in \mc{C}$ be composite. Then we have $a = np_a \geq p_a^2$ for some integer $n|a$. Hence,
\begin{align*}
   \sum\limits_{a\in \mc{C}} \frac{1}{a} \leq  \sum\limits_{p_a\in \mc{P}_\mc{C}} \frac{1}{p_a^2} < \sum\limits_{p\in \mb{P}} \frac{1}{p^2}<\infty,
\end{align*}
and
\begin{align*}
   \sum\limits_{a\in \mc{P}} \frac{1}{a} = \sum\limits_{a\in \mc{A}} \frac{1}{a} - \sum\limits_{a\in \mc{C}} \frac{1}{a} = \infty.
\end{align*}
Let $\mb{P}^2 = \{p^2: p \in \mb{P}\}$. Then we have $\lla\mb{P}^2\rra = \{n^2: n \in \N\}$ and
\begin{align*}
    |\lla\mb{P}^2\rra_x| \leq x^{1/2}.
\end{align*}
Let $J \in \N$ and $m = \prod\limits_{j = 1}^Jc_j^{\nu_j}\in \lla\mc{C}\rra_x$ with $c_j \in \mc{C}, \nu_j \in \N$ for all $j \leq J$. We consider a map $\iota: \lla\mc{C}\rra_x \to \lla\mb{P}^2\rra_x$ where
\begin{align}\lb{injective}
    \iota(m=\prod\limits_{j = 1}^Jc_j^{\nu_j}) = \prod\limits_{j = 1}^Jp_{c_j}^{2\nu_j} = (\prod\limits_{j = 1}^Jp_{c_j}^{\nu_j})^2.
\end{align}
It is clear that $\iota$ is injective and hence
\begin{align*}
    |\lla\mc{C}\rra_x| \leq |\lla\mb{P}^2\rra_x| \leq x^{1/2}.
\end{align*}

\end{proof}

Before giving the proof of Theorem \ref{thm1.2}, we will briefly explain how we take advantage of \ef{sparse}. Let $n \in \N$. By the decomposition \ef{decom}, we can split $\lambda_\mc{A}(n)$ into two functions, $\lambda_\mc{C}$ and the multiplicative function $\lambda_\mc{P}$, for which,
\begin{align*}
    \lambda_\mc{A}(n) = \lambda_\mc{A}(n_{\mc{C}})\lambda_\mc{A}(n/n_{\mc{C}}) = \lambda_\mc{C}(n_{\mc{C}})\lambda_\mc{P}(n/n_{\mc{C}})
\end{align*}
where $n_{\mc{C}} \in \lla\mc{C}\rra$ divides $n$ and $n/n_{\mc{C}}$ has no divisors in $\mc{C}$. Equivalently, $\lambda_\mc{A}(n)$ can be expressed as
\begin{align}\lb{newexpress}
    \lambda_\mc{A}(n) = \sum\limits_{\substack{n_{\mc{C}} | n, \\n_{\mc{C}} \in \lla\mc{C}\rra}}\lambda_\mc{C}(n_{\mc{C}})\tilde{\lambda}_\mc{P}(n/n_{\mc{C}})
\end{align}
where 
\begin{align}\lb{defp}
    \tilde{\lambda}_\mc{P}(n) = \lambda_\mc{P}(n) \prod\limits_{c \in \mc{C}}(1-\mathbbm{1}_{c | n}).
\end{align}
Roughly speaking, \ef{sparse} allows us to reduce the averages of $\lambda_\mc{A}$ (or correlations of $\lambda_\mc{A}$) to a sum over $n_{\mc{C}} \in \lla\mc{C}\rra_x$ of the averages of the multiplicative function $\lambda_\mc{P}$ (or correlations of $\lambda_\mc{P}$) and we can control the averages of $\lambda_\mc{P}$ using multiplicativity. 

\begin{proof}[Proof of Theorem 1.1 (a)]
By \ef{newexpress}, when summing over $n$, we have
\begin{align}\lb{w1}
    \frac{1}{x}\sum\limits_{n \leq x} \lambda_\mc{A}(n) 
    &= \frac{1}{x}\sum\limits_{n_{\mc{C}} \in \lla\mc{C}\rra_x}\lambda_\mc{C}(n_{\mc{C}}) \sum\limits_{m \leq x/n_{\mc{C}}}\tilde{\lambda}_\mc{P}(m)\nonumber\\
    &=\frac{1}{x}\sum\limits_{n_{\mc{C}} \in \lla\mc{C}\rra_{x^{2/3}}}\lambda_\mc{C}(n_{\mc{C}}) \sum\limits_{m \leq x/n_{\mc{C}}}\tilde{\lambda}_\mc{P}(m) + O(x^{-1/6}).
\end{align}
The second equality holds, since by \ef{index}, 
\begin{align*}
    \sum\limits_{\substack{x^{2/3}< n_{\mc{C}} \leq x,\\n_{\mc{C}} \in \lla\mc{C}\rra}}\lambda_\mc{C}(n_{\mc{C}}) \sum\limits_{m \leq x/n_{\mc{C}}}\tilde{\lambda}_\mc{P}(m) \leq \sum\limits_{\substack{n_{\mc{C}} \leq x,\\n_{\mc{C}} \in \lla\mc{C}\rra}}\sum\limits_{m \leq x^{1/3}}1 
    \leq |\lla\mc{C}\rra_x| \cdot x^{1/3}\leq x^{5/6}.
\end{align*}
Let $C = \prod\limits_{c \in \mc{C}, c \leq x}c$. By substituting \ef{defp} into \ef{w1}, we have
\begin{align}\lb{h1}
    \frac{1}{x}\sum\limits_{n \leq x} \lambda_\mc{A}(n) &=\frac{1}{x}\sum\limits_{n_{\mc{C}} \in \lla\mc{C}\rra_{x^{2/3}}}\lambda_\mc{C}(n_{\mc{C}}) \sum\limits_{m \leq x/n_{\mc{C}}}\lambda_\mc{P}(m) \prod\limits_{c \in \mc{C}}(1-\mathbbm{1}_{c | n}) + O(x^{-1/6})\nonumber \\
    &= \frac{1}{x}\sum\limits_{n_{\mc{C}} \in \lla\mc{C}\rra_{x^{2/3}}}\lambda_\mc{C}(n_{\mc{C}}) \sum\limits_{\substack{d | C,\\d \in \lla\mc{C}\rra_x}}(-1)^{\omega_{\mc{C}}(d)}\lambda_\mc{P}(d)\sum\limits_{m \leq x/n_{\mc{C}}d}\lambda_\mc{P}(m) + O(x^{-1/6})\nonumber \\
    &\ll \frac{1}{x}\sum\limits_{n_{\mc{C}} \in \lla\mc{C}\rra_{x^{2/3}}}\sum\limits_{\substack{d | C,\\d \in \lla\mc{C}\rra_x}}\Big|\sum\limits_{m \leq x/n_{\mc{C}}d}\lambda_\mc{P}(m)\Big| + x^{-1/6}
\end{align}
By \ef{sparse}, we have
\begin{align}\lb{sx}
    I(x) = \sum\limits_{n \in \lla\mc{C}\rra_x} \frac{1}{n} \ll \prod\limits_{c \in \mc{C}} \sum\limits_{i = 0}^\infty \frac{1}{c^i} \ll \prod\limits_{c \in \mc{C}}(1 + \frac{1}{c-1}) = O(1),
\end{align}
and hence
\begin{align*}
    \sum\limits_{\substack{d | C,\\d \in \lla\mc{C}\rra_x,\\d> x^{1/4}}} \frac{1}{d}\ll x^{-1/8}\sum\limits_{\substack{d | C,\\d \in \lla\mc{C}\rra_x,\\d> x^{1/8}}} \frac{1}{d} \ll x^{-1/8}I(x) \ll x^{-1/8}.
\end{align*}
The error of truncating the inner sum in \ef{h1} to $d \in \lla\mc{C}\rra_{x^{1/4}}$ is therefore
\begin{align*}
    \frac{1}{x}\sum\limits_{n_{\mc{C}} \in \lla\mc{C}\rra_{x^{2/3}}}\sum\limits_{\substack{d | C,\\d \in \lla\mc{C}\rra_x,\\d> x^{1/4}}}\Big|\sum\limits_{m \leq x/n_{\mc{C}}d}\lambda_\mc{P}(m)\Big| 
    &\leq \sum\limits_{n_{\mc{C}} \in \lla\mc{C}\rra_{x^{2/3}}}\frac{1}{n_{\mc{C}}}\sum\limits_{\substack{d | C,\\d \in \lla\mc{C}\rra_x\\ d > x^{1/4}}}\frac{1}{d} \ll x^{-1/8}.
\end{align*}
Therefore, 
\begin{align*}
     \frac{1}{x}\sum\limits_{n \leq x} \lambda_\mc{A}(n) \ll \frac{1}{x}\sum\limits_{n_{\mc{C}} \in \lla\mc{C}\rra_{x^{2/3}}}\sum\limits_{\substack{d | C,\\d \in \lla\mc{C}\rra_{x^{1/4}}}}\Big|\sum\limits_{m \leq x/n_{\mc{C}}d}\lambda_\mc{P}(m)\Big| + x^{-1/8}.
\end{align*}
As $n_{\mc{C}}d \leq x^{11/12}$ for all $n_{\mc{C}} \in \lla\mc{C}\rra_{x^{2/3}}, d \in \lla\mc{C}\rra_{x^{1/4}}$, by a theorem of Hall and Tenenbaum \cite[III.~4.3 on p.~345]{tenenbaum2015introduction},
\begin{align}\lb{hall}
   \sum\limits_{m \leq x/n_{\mc{C}}d}\lambda_\mc{P}(m) \ll \frac{x}{n_{\mc{C}}d} \exp \Bigg(-2K \sum\limits_{\substack{p \leq x/n_{\mc{C}}d, \\ p \in \mc{P}}} \frac{1}{p}\Bigg) \ll \frac{x}{n_{\mc{C}}d} \exp \Bigg(-2K\sum\limits_{\substack{p \leq x^{1/12}, \\ p \in \mc{P}}} \frac{1}{p}\Bigg),
\end{align}
where $K > 0$ is an absolute constant. By \ef{sparse}, \ef{sx} and \ef{hall}, we obtain
\begin{align*}
     \frac{1}{x}\sum\limits_{n \leq x} \lambda_\mc{A}(n) &\ll \exp \Bigg(-2K\sum\limits_{\substack{p \leq x^{1/12}, \\ p \in \mc{P}}} \frac{1}{p}\Bigg)\sum\limits_{n_{\mc{C}} \in \lla\mc{C}\rra_{x^{2/3}}}\frac{1}{n_{\mc{C}}}\sum\limits_{\substack{d | C,\\d \in \lla\mc{C}\rra_{x^{1/4}}}}\frac{1}{d} + x^{-1/8}\\
     &\ll \exp \Bigg(-2K\sum\limits_{\substack{p \leq x^{1/12}, \\ p \in \mc{P}}} \frac{1}{p}\Bigg) I(x)^2+ x^{-1/8}= o(1) \text{ as } x \to \infty.
\end{align*}

\end{proof}

\begin{proof}[Proof of Theorem 1.1 (b)]
Let
\begin{align}\lb{lazysum}
    S_k(x)= \frac{1}{x}\sum\limits_{n\leq x}\lambda_\mc{A}(a_1n+h_1)\cdots\lambda_\mc{A}(a_kn+h_k),
\end{align}
and $X = \max\limits_{1 \leq i \leq k} (a_ix+h_i)$. Substituting \ef{newexpress} into \ef{lazysum}, we have
\begin{align}\lb{beforesum}
    S_k(x)= \sum\limits_{n_{\mc{C}_1}, ..., n_{\mc{C}_k} \in \lla\mc{C}\rra_X}\prod_{i = 1}^k\lambda_\mc{C}(n_{\mc{C}_i})\cdot S(x;n_{\mc{C}_1}, ...,n_{\mc{C}_k}),
\end{align}
where
\begin{align*}
     S(x;n_{\mc{C}_1}, ...,n_{\mc{C}_k}) =\frac{1}{x} \sum\limits_{\substack{n \leq x \\
    n_{\mc{C}_i}|a_in + h_i, \forall i \leq k}}\prod_{i = 1}^k\tilde{\lambda}_\mc{P}\Bigg(\frac{a_in+h_i}{n_{\mc{C}_i}}\Bigg).
\end{align*}
Let $0<T < (\log x)^{1/4k}$ depend on $x$ and tend to $\infty$ slowly as $x \to \infty$ which will be determined later. We want to truncate the sum in \ef{beforesum} over $n_{\mc{C}_i} \in \lla\mc{C}\rra_X$ to $n_{\mc{C}_i} \in \lla\mc{C}\rra_T$ for all $1 \leq i \leq k$ with admissible error. In order to achieve this, we will truncate $n_{\mc{C}_i} \in \lla\mc{C}\rra_X$ to $n_{\mc{C}_i}\in \lla\mc{C}\rra_z$ for some $z < x$ for all $1 \leq i \leq k$ and then to $T$. Let $\mc{R}_1(x), \mc{R}_2(x)$ denote the error term of these two truncations, respectively. We will show that $\mc{R}_1(x), \mc{R}_2(x) = o(1)$ as $x \to \infty$. 

\subsection{Truncation from $n_{\mc{C}_i} \in \lla\mc{C}\rra_X$ to $n_{\mc{C}_i} \in \lla\mc{C}\rra_T, \forall i \leq k$}
Let $g_i$ denote the greatest common divisor of $a_i, h_i, n_{\mc{C}_i}$ for $1 \leq i \leq k$ and $\tilde{n}_{\mc{C}_i} = n_{\mc{C}_i}/g_i$. If $a_in + h_i \equiv 0$ (mod $n_{\mc{C}_i}$) for some $i \leq k$, then $n$ should satisfy $n \equiv r_i$ (mod $\tilde{n}_{\mc{C}_i}$) for some integer $0 \leq r_i < \tilde{n}_{\mc{C}_i}$. Let 
\begin{align*}
     N=N(n_{\mc{C}_1}, ..., n_{\mc{C}_k}) = [n_{\mc{C}_1}, ..., n_{\mc{C}_k}].
\end{align*} 
Without loss of generality, we may assume $n_{\mc{C}_1}>z$ and hence $N \geq n_{\mc{C}_1} > z $. Then we have
\begin{align*}
    \mc{R}_1(x) &= \sum\limits_{\substack{n_{\mc{C}_1}, ..., n_{\mc{C}_k} \in \lla\mc{C}\rra_X,\\ n_{\mc{C}_1} > z}}\prod_{i = 1}^k\lambda_\mc{C}(n_{\mc{C}_i})\cdot S(x;n_{\mc{C}_1}, ...,n_{\mc{C}_k}) \\
    &\ll \frac{1}{x}\sum\limits_{\substack{n_{\mc{C}_1}, ..., n_{\mc{C}_k} \in \lla\mc{C}\rra_X,\\ N > z}}\sum\limits_{n \leq x}\prod_{i = 1}^k\mathbbm{1}_{n \equiv r_i (\tilde{n}_{\mc{C}_i})} \ll \frac{1}{x}\sum\limits_{\substack{n_{\mc{C}_1}, ..., n_{\mc{C}_k} \in \lla\mc{C}\rra_X,\\ N > z}}\sum\limits_{n \leq x}\prod_{i = 1}^k\sum\limits_{j = 0}^{g_i - 1}\mathbbm{1}_{n \equiv r_i+j \tilde{n}_{\mc{C}_i}(n_{\mc{C}_i})}.
\end{align*}
Let $0 \leq \tilde{r}_i < n_{\mc{C}_i}$ be an integer for $i \leq k$. By \ef{index} and \cite[Lemma 2.5]{gorodetsky2023squarefrees}, \footnote{Lemma 2.5 can be proven assuming $|\lla B\rra_x| \leq x^{\alpha+o(1)}$ for some $\alpha \in (0, 1)$ rather than that $\lla B\rra$ has an index.} we have
\begin{align}\lb{1trunc}
    \sum\limits_{\substack{n_{\mc{C}_1}, ..., n_{\mc{C}_k} \in \lla\mc{C}\rra_X,\\ N > z}}\sum\limits_{n \leq x}\prod_{i = 1}^k\mathbbm{1}_{n \equiv \tilde{r}_i (n_{\mc{C}_i})} \ll X^{o(1)}(Xz^{-1/2+o(1)}+X^{\frac{2}{3}}) \ll_{a_i, h_i} x^{o(1)}(xz^{-1/2+o(1)}+x^{\frac{2}{3}}).
\end{align}
Note that the size of \ef{1trunc} is independent of the choice of $\tilde{r}_i$ for all $i \leq k$. Since $g_i \leq \min\{a_i, h_i\}$ for all $i \leq k$ are fixed, we have
\begin{align*}
    \mc{R}_1(x) \ll \prod_{i = 1}^kg_i \sum\limits_{\substack{n_{\mc{C}_1}, ..., n_{\mc{C}_k} \in \lla\mc{C}\rra_X,\\ N > z}}\sum\limits_{n \leq x}\prod_{i = 1}^k\mathbbm{1}_{n \equiv \tilde{r}_i (n_{\mc{C}_i})}
    \ll x^{o(1)}(xz^{-1/2+o(1)}+x^{\frac{2}{3}}).
\end{align*}
Let $z = x^{1/10k}$. We obtain
\begin{align*}
    \mc{R}_1(x) = \sum\limits_{\substack{n_{\mc{C}_1}, ..., n_{\mc{C}_k} \in \lla\mc{C}\rra_X,\\ n_{\mc{C}_1} > z}}\prod_{i = 1}^k\lambda_\mc{C}(n_{\mc{C}_i})\cdot S(x;n_{\mc{C}_1}, ...,n_{\mc{C}_k}) = o(1) \text{ as } x \to \infty.
\end{align*}
Next, we will truncate $n_{\mc{C}_i} \in \lla\mc{C}\rra_z$ to $\lla\mc{C}\rra_T$. Without loss of generality, we assume $T < n_{\mc{C}_1} \leq x^{1/10k}$ and hence $ T < n_{\mc{C}_1} \leq N \leq x^{1/10} $. If $n_{\mc{C}_i}|a_in + h_i$ for all $ i \leq k$, solutions $n$ are of the form $n \equiv q$ mod $N$ where $q \leq N \in \N$ only depending on $n_{\mc{C}_i}, a_i, h_i$ for $1 \leq i \leq k$. The number of possibilities for the class $q$ (mod $N$) is at most $\prod_{i = 1}^kg_i$ which is fixed. Let $Q$ be the class of $q$ (mod $N$) where $|S(x;n_{\mc{C}_1}, ...,n_{\mc{C}_k})|$ takes the maximum. So when $n_{\mc{C}_i}|a_in + h_i, \forall i \leq k$ is solvable and $N \leq x^{1/10}$, we have
\begin{align}\lb{defs}
    S(x;n_{\mc{C}_1}, ...,n_{\mc{C}_k}) &\ll \frac{1}{x}\sum\limits_{m\leq x/N}\prod\limits_{i = 1}^k\tilde{\lambda}_\mc{P}\Bigg(\frac{a_i (Q+mN) + h_i}{n_{\mc{C}_i}}\Bigg) + \frac{1}{x}\nonumber\\
    &\ll \frac{1}{x} \sum\limits_{m\leq x/N}\prod\limits_{i = 1}^k\tilde{\lambda}_\mc{P}(A_im + H_i)+ \frac{1}{x} \text{ where } A_i = \frac{a_iN}{n_{\mc{C}_i}}, H_i=\frac{a_iQ+h_i}{n_{\mc{C}_i}}.
\end{align}
By bounding $S(x;n_{\mc{C}_1}, ...,n_{\mc{C}_k}) \ll \frac{1}{[n_{\mc{C}_1}, ...,n_{\mc{C}_k}]}$ trivially,
we have
\begin{align}\lb{bully}
     \mc{R}_2(x) = \sum\limits_{\substack{n_{\mc{C}_1}, ..., n_{\mc{C}_k} \in \lla\mc{C}\rra_z,\\ n_{\mc{C}_1} > T}}\prod_{i = 1}^k\lambda_\mc{C}(n_{\mc{C}_i})\cdot S(x;n_{\mc{C}_1}, ...,n_{\mc{C}_k}) \ll \sum\limits_{\substack{n_{\mc{C}_1}, ..., n_{\mc{C}_k} \in \lla\mc{C}\rra_z,\\n_{\mc{C}_1} > T}} \frac{1}{[n_{\mc{C}_1}, ...,n_{\mc{C}_k}]}.
\end{align}
Let $B_i \in \lla\mc{C}\rra_z $ be the greatest integer such that $B_i | n_{\mc{C}_1}$ and $B_i | n_{\mc{C}_i}$, and $e_i = n_{\mc{C}_i}/B_i $  for $2 \leq i \leq k$. Then we have
\begin{align}\lb{change}
    N = [n_{\mc{C}_1}, ...,n_{\mc{C}_k}] = [n_{\mc{C}_1}, B_2e_2,.., B_ke_k] = n_{\mc{C}_1}[e_2, ..., e_k],
\end{align}
Moreover, since $n_{\mc{C}_i} | a_in + h_i$ for all $i \leq k$, we have $B_i | a_1n+h_1$ and $B_i | a_in+h_i$, and hence
\begin{align}\lb{sized}
    B_i | a_1h_i - a_ih_1 \text{ for all $ 2 \leq i \leq k$.}
\end{align}
Let $\mc{D}_i$ be the set of all possibilities of $B_i$ for $2 \leq i \leq k$ and $\mc{D} = \bigcup\limits_{i = 2}^k \mc{D}_i$. By \ef{sized}, we have $|\mc{D}| = O(1)$ as $a_i, h_i$ are fixed for all $ i \leq k$. Applying \ef{change} to \ef{bully}, we have
\begin{align*}
    \mc{R}_2(x) \ll \sum\limits_{\substack{n_{\mc{C}_1}, ..., n_{\mc{C}_k} \in \lla\mc{C}\rra_z,\\ n_{\mc{C}_1}>T}} \frac{1}{[n_{\mc{C}_1}, ...,n_{\mc{C}_k}]} 
    &\ll |\mc{D}|^k\sum\limits_{\substack{n_{\mc{C}_1} \in \lla\mc{C}\rra_z, \\ n_{\mc{C}_1}>T}}\frac{1}{n_{\mc{C}_1}}\sum\limits_{e_2, ..., e_k\in \lla\mc{C}\rra_z}\frac{1}{[e_2, ..., e_k]}.
\end{align*}
Since by \ef{sparse} for any integer $l \geq 1$
\begin{align} \lb{boundsum}
    \sum\limits_{n_{1}, ..., n_{l} \in \lla\mc{C}\rra_x}\frac{1}{[n_{1},..., n_{l}]} &\leq \sum\limits_{n \in \lla\mc{C}\rra_{x^l}}\frac{\Big(\sum\limits_{m|n, m \in \lla\mc{C}\rra}1\Big)^l}{n} \leq \prod\limits_{c \in \mc{C}, c \leq x}\Bigg(1 + \sum\limits_{\nu = 1}^\infty \frac{(\nu+1)^l}{c^\nu}\Bigg)\nonumber \\
    &\ll \prod\limits_{c \in \mc{C}}\Bigg(1 + \frac{2^l} {c}\Bigg) \ll \prod\limits_{c \in \mc{C}}\Bigg(1 + \frac{1} {c}\Bigg)^{2^l} \ll 1
\end{align}
 and by \ef{sx}
\begin{align*}
    \sum\limits_{\substack{n_{\mc{C}_1} \in \lla\mc{C}\rra_z, \\ n_{\mc{C}_1}>T}}\frac{1}{n_{\mc{C}_1}} \ll T^{-1/2}I(x) \ll T^{-1/2},
\end{align*}
thus we obtain
\begin{align*}
    \mc{R}_2(x) \ll T^{-1/2} |\mc{D}|^k \sum\limits_{e_{2}, ..., e_{k} \in \lla\mc{C}\rra_x}\frac{1}{[e_{2},..., e_{k}]} \ll T^{-1/2} = o(1) \text{ as } x \to \infty.
\end{align*}
As a result, we can truncate the sum in \ef{beforesum} with admissible error, for which
\begin{align}\lb{aftersum}
    S_k(x) = \sum\limits_{\substack{n_{\mc{C}_1}, ..., n_{\mc{C}_k} \in \lla\mc{C}\rra_T}}\prod_{i = 1}^k\lambda_\mc{C}(n_{\mc{C}_i})\cdot S(x;n_{\mc{C}_1}, ...,n_{\mc{C}_k}) + o(1).
\end{align}
\subsection{Estimating the truncated sum} After truncation, our work reduces to estimating \ef{aftersum}. Substituting \ef{defp} into \ef{defs}, we have
\begin{align}\lb{again}
    |S(x;n_{\mc{C}_1}, ...,n_{\mc{C}_k})| &\ll \frac{1}{x}\Bigg|\sum\limits_{m\leq x/N}\prod\limits_{i = 1}^k\lambda_\mc{P}(A_im + H_i)\prod\limits_{c \in \mc{C}}(1-\mathbbm{1}_{c | A_im + H_i})\Bigg|+ \frac{1}{x}\nonumber \\
    &\ll \sum\limits_{d_1, ...,d_k \in \lla\mc{C}\rra_X} \Bigg|\sum\limits_{\substack{m\leq x/N,\\ d_i | A_im + H_i, \forall i\leq k}}\prod\limits_{i = 1}^k\lambda_\mc{P}(A_im + H_i)\Bigg|+ \frac{1}{x}.
\end{align}
Similarly, by the previous argument about truncation from $n_{\mc{C}_i} \in \lla\mc{C}\rra_X$ to $\lla\mc{C}\rra_T$ where $T \leq (\log x)^{1/4k}$, we can also truncate the sum in \ef{again} to $d_i \in \lla\mc{C}\rra_T $ for all $i \leq k$ with admissible error. Let 
\begin{align*}
     D= D(d_1, ..., d_k) = [d_1, ..., d_k] \leq T^k \leq (\log x)^{1/4}.
\end{align*}
If $d_i|A_im + H_i$ for all $ i \leq k$, solutions $m$ are of the form $m \equiv M$ mod $D$ where $0 \leq M < D \in \N$ depending on $n_{\mc{C}_i}, a_i, h_i, d_i$ only. Let $M(T)$ denote the number of possibilities for the class $M$  (mod $D$), which depends on $T$. Let $M'$ be the class $M$ (mod $D$) where 
\begin{align*}
    \Bigg|\sum\limits_{m\leq x/ND}\prod\limits_{i = 1}^k\lambda_\mc{P}\Bigg(\frac{A_i (M+mD) + H_i}{d_i}\Bigg)\Bigg|
\end{align*}
takes the maximum. When $d_i|A_im + H_i, \forall i \leq k$ is solvable, we have
\begin{align}\lb{bean}
    |S(x;n_{\mc{C}_1}, ...,n_{\mc{C}_k})|&\ll\frac{M(T)}{x}\sum\limits_{d_1, ...,d_k \in \lla\mc{C}\rra_T} \Bigg|\sum\limits_{m\leq x/ND}\prod\limits_{i = 1}^k\lambda_\mc{P}(A'_{i}m + H'_{i})\Bigg|+ \frac{T^kM(T)}{x},
\end{align}
where
\begin{align*}
    A'_{i} = \frac{A_iD}{d_i}=\frac{a_iND}{n_{\mc{C}_i}d_i}, H'_{i} = \frac{A_iM'+ H_i}{d_i} = \frac{a_iNM'+a_iQ+h_i}{n_{\mc{C}_i}d_i}.
\end{align*}
Let $X' = x/ND$. We have $x^{1/2} < X' \leq x $ since $ND \leq T^{2k} \leq (\log x)^{1/2}$ are small. Let $\e \in (1/\log \log x, 1/2)$ and $\e' \in [\e, 2\e]$ such that $(X')^{\e'} = x^\e$. Since
\begin{align*}
   \mb{D} (\lambda_\mc{P}, 1; x^\e, x) = \sum\limits_{\substack{x^\e<p\leq x,\\p\in \mc{P}}}\frac{2}{p} \leq \e,\text{ and } \mb{D}(\lambda_\mc{P}, 1; x^\e)=\sum\limits_{\substack{p \leq x^\e,\\p \in \mc{P}}} \frac{2}{p}\geq 1/\e \text{ as } x \to \infty,
\end{align*}
by \cite[Proposition 4.3 and Theorem 4.1]{klurman2023elliott} we have
\begin{align}\lb{goodbound}
    \Bigg|\sum\limits_{m\leq X'}\prod\limits_{i = 1}^k\lambda_\mc{P}(A'_{i}m + H'_{i})\bigg| &\ll C(T)X'R \ll C(T) x/ND \Big(\log \frac{1}{\e}\Big)^{1/2}\e,
\end{align}
where
\begin{align*}
    R = \Big(\log \frac{1}{\e}\Big)^{1/2} \mb{D}(\lambda_\mc{P}, 1;x^\e, x)+&\exp(-\mb{D}(\lambda_\mc{P}, 1;x^\e)^2)+\exp(-\frac{1}{8k^2\e'})
\end{align*}
and the implicit constant $C = C(T)$ in \ef{goodbound} depends on $T$. Also, by \ef{boundsum}, we have
\begin{align}\lb{qdbound}
    \sum\limits_{d_1, ...,d_k \in \lla\mc{C}\rra_T}\frac{1}{[d_1, ...,  d_k]} \leq \sum\limits_{n_{1}, ..., n_{k} \in \lla\mc{C}\rra_x}\frac{1}{[n_{1},..., n_{k}]}= O(1).
\end{align}
Applying \ef{bean}, \ef{goodbound} and \ef{qdbound} to \ef{aftersum} and recalling the definitions of $N$ and $D$, we obtain
\begin{align*}
    S_k(x)
    &\ll \sum\limits_{n_{\mc{C}_1}, ..., n_{\mc{C}_k} \in \lla\mc{C}\rra_T} |S(x;n_{\mc{C}_1}, ...,n_{\mc{C}_k})| + o(1)\\
    &\ll \frac{M(T)}{x}\sum\limits_{n_{\mc{C}_1}, ..., n_{\mc{C}_k} \in \lla\mc{C}\rra_T}\sum\limits_{d_1, ...,d_k \in \lla\mc{C}\rra_T} \Bigg|\sum\limits_{m\leq x/ND}\prod\limits_{i = 1}^k\lambda_\mc{P}(A'_{i}m + H'_{i})\Bigg|+ \frac{T^kM(T)}{x}\\
    &\ll M(T)C(T)\Big(\log \frac{1}{\e}\Big)^{1/2}\e\sum\limits_{n_{\mc{C}_1}, ..., n_{\mc{C}_k} \in \lla\mc{C}\rra_T}\frac{1}{[n_{\mc{C}_1}, ...,  n_{\mc{C}_k}]}\sum\limits_{d_1, ...,d_k \in \lla\mc{C}\rra_T}\frac{1}{[d_1, ...,  d_k]}+ \frac{T^kM(T)}{x}\\
    &\ll M(T)C(T)\Big(\log \frac{1}{\e}\Big)^{1/2}\e+ \frac{T^kM(T)}{x}.
\end{align*}
Let $T = \eta(x) < (\log x)^{1/4k}$ be a function tending to $\infty$ slowly as $x \to \infty$. Let $\e = \mu(T)= \mu(\eta(x))> 1/ \log \log x $ be a function tending to $0$ slowly as $x \to \infty$ such that $M(T)C(T)\Big(\log \frac{1}{\e}\Big)^{1/2}\e$ tends to $0$ as $x \to \infty$. Then we obtain
\begin{align*}
    \lim_{x\to\infty}\Big|\frac{1}{x}\sum\limits_{n\leq x}\prod_{i = 1}^k\lambda_\mc{A}(a_in+h_i)\Big| = 0
\end{align*}
as desired.
\end{proof}
\begin{acknowledgements}
    I would like to thank my supervisor Sacha Mangerel for fruitful discussions and constructive comments.
\end{acknowledgements}

\bibliographystyle{plain}
\addcontentsline{toc}{chapter}{Bibliography}
\bibliography{lambda}

\end{document}